\documentclass[12pt,oneside,reqno]{amsart}%
\makeatletter
\usepackage{amsfonts}
\usepackage{amsmath}
\usepackage{amssymb}
\usepackage{graphicx}
\usepackage{booktabs}
\usepackage{subcaption}
\usepackage{epstopdf} % with/without is the same
\usepackage{version}
\usepackage{bm}
\usepackage{enumerate}
\usepackage{epsfig}
\usepackage{caption}
\usepackage{xcolor}
\usepackage{setspace}
\usepackage{amsthm}
\usepackage{caption}
\DeclareCaptionLabelFormat{cont}{#1~#2\alph{ContinuedFloat}}
\usepackage{footnote}
\theoremstyle{plain}
\newtheorem{theorem}{Theorem}[section]
\newtheorem{corollary}[theorem]{Corollary}

\newtheorem{proposition}[theorem]{Proposition}
\newtheorem{definition}[theorem]{Definition}

\newtheorem{assumption}[theorem]{Assumption}
\newtheorem{remark}[theorem]{Remark}

\newtheorem{observation}[theorem]{Observation}
\newtheorem*{notation*}{Notation}
\numberwithin{equation}{section}

\excludeversion{comment1}

\begin{document}
\title{Understanding and Resolving Singularities in 3D Dirichlet Boundary Problems}
\author{  David Levin  }

\footnotetext[1]{David Levin, School of Mathematical Sciences, Tel-Aviv University, Israel}

\begin{abstract}

We introduce a two-phase approximation method designed to resolve singularities in three-dimensional harmonic Dirichlet problems.  The approach utilizes the classical Green's function representation, decomposing the function into its singular and regular components.

The singular phase employs Green’s formula with the singular part, for which we show that it induces the necessary singularities in the solution. The regular phase then introduces a smooth correction to recover the remaining regular part of the solution.
The construction employs high-order quadrature rules in the first phase, followed by collocation with a suitable harmonic basis in the second.

\end{abstract}
\maketitle

\section{Introduction}

\section*{Problem Setup}

Let \(\Omega \subset \mathbb{R}^3\) be a bounded domain with boundary \(\partial \Omega\), and consider the Dirichlet problem:
\[
\begin{cases}
\Delta u = 0, & \text{in } \Omega, \\
u = f, & \text{on } \partial \Omega.
\end{cases}
\]
The boundary $\partial\Omega$ may be smooth or non-smooth, and 
the boundary data \(f\) may be continuous on \(\partial \Omega\), piecewise continuous, or in \(L^2(\partial \Omega)\).

A weak solution \(u \in H^1(\Omega)\) exists under mild assumptions on the boundary data (e.g., \(f \in L^2(\partial \Omega)\)), and the solution is harmonic; hence, it is infinitely differentiable in the interior of \(\Omega\).
However, the solution may exhibit singular behavior near geometric singularities of the domain, such as the corners of $\partial\Omega$, as well as near discontinuities in the boundary data $f$. In particular, the gradient \(\nabla u\) may become unbounded in the vicinity of these singularities.

One of the most widely used practical methods for solving elliptic boundary value problems is the Finite Element Method (FEM); see Babuška \cite{Babuska}. However, in problems governed by the Laplace and Poisson equations, the spatial derivatives of the potential field may become unbounded near corners and edges of the domain. As a result, conventional finite elements often fail to capture the solution accurately in the vicinity of such singularities. This limitation has motivated significant research into the development and implementation of specialized singular elements for two-dimensional finite element analysis, specifically designed to resolve corner and edge singularities in potential problems: see, for example, \cite{Cost, Gris, Kozlov, Mukh, Ong}. However, in contrast to the two-dimensional case, the explicit characterization of singularities in solutions to the three-dimensional harmonic problem near corners and edges remains incomplete.

In the two--dimensional setting, a corrected collocation approach
for treating corner singularities was proposed in \cite{Levin1980}.
The method is tailored to geometric configurations
that admit a singular--regular decomposition of the corresponding Green's function.
The extension of this approach to higher dimensions has,
to the best of our knowledge, remained unresolved.
The purpose of the present paper is to address this problem.

The results presented here extend the methodology of \cite{Levin1980}
to three dimensions.
While the applicability of the method is still confined to certain
geometric structures, it already covers important examples,
including cubes and cylinders.
Furthermore, the analysis developed in this work
provides a framework for understanding the singular behavior
of solutions in more general geometries.

The paper is organized as follows. 
In Section~2 we recall Green’s identity and review explicit 
representations of the Dirichlet Green’s function in geometries 
of interest, including the half–space, the unit cube, 
and cylindrical domains. These examples motivate the 
singular–regular decomposition employed throughout the paper. 
In Section~3 we formulate the S–R method in a general setting, 
establish the decomposition of the solution into singular 
and harmonic components, and describe the resulting two–phase 
approximation procedure. Particular attention is given to the 
implementation for the unit cube, including the computation 
of the singular integral contribution, the construction of 
harmonic approximants via both polynomial interpolation and 
the method of fundamental solutions, and the treatment of 
nearly singular boundary integrals. 
Section~4 presents numerical experiments that assess stability, 
conditioning, and accuracy of the proposed method, with 
emphasis on corner–induced singularities and high–order 
evaluation of the solution near the boundary.

\section{Preliminaries - Green's identity and Green's function}

Let $\Omega \subset \mathbb{R}^3$ be a bounded domain with sufficiently smooth boundary $\partial\Omega$. 
We consider the harmonic Dirichlet problem
\begin{equation}
\begin{cases}
-\Delta u(x) = 0, & x \in \Omega, \\
u(x) = f(x), & x \in \partial\Omega .
\end{cases}
\end{equation}
Given boundary data $f$ on $\partial\Omega$, the objective is to determine a function 
$u : \overline{\Omega} \to \mathbb{R}$ that is harmonic in $\Omega$ and satisfies the prescribed boundary values. 
Under standard assumptions on $f$ (for example, $f \in L^2(\partial\Omega)$), the problem admits a weak solution 
$u \in H^1(\Omega)$, which is smooth in the interior of $\Omega$.

In 1828, George Green presented his second integral identity:

\[
\int_{\Omega} \big( u \,\Delta v - v \,\Delta u \big)\, dV
  = \int_{\partial\Omega} 
      \left( 
          u \,\frac{\partial v}{\partial n}
          - v \,\frac{\partial u}{\partial n}
      \right) dS .
\]
Building on this identity, Green introduced what is now known as Green's function, providing an explicit integral representation for solutions of Laplace’s equation.

Let $\Omega \subset \mathbb{R}^3$ be a closed domain with a piecewise smooth boundary 
$\partial\Omega$.  
The Green's function $G(x,y)$ for the Laplace operator in $\Omega$ is defined as 
the function satisfying
\begin{equation}
\label{eq:Green-def}
\begin{cases}
\Delta_y G(x,y) = -\delta(x-y), & y \in \Omega, \\[4pt]
G(x,y) = 0, & y \in \partial\Omega .
\end{cases}
\end{equation}

Here $\delta$ denotes the Dirac delta distribution, and $\Delta_y$ is the Laplacian 
with respect to the variable $y$.

Given boundary data $u = f$ on $\partial\Omega$, the harmonic Dirichlet problem
\begin{equation}
\Delta u = 0 \quad \text{in } \Omega, 
\qquad u = f \quad \text{on } \partial\Omega
\end{equation}
has the solution representation \cite{Green1828}
\begin{equation}
\label{eq:GR}
u(x) 
= \int_{\partial\Omega}
    \frac{\partial G}{\partial n_y}(x,y)f(y)\, ds_y,
\qquad x \in \Omega,
\end{equation}
where $\dfrac{\partial G}{\partial n_y}$ denotes differentiation with respect 
to the outward unit normal at $y \in \partial\Omega$.

{\it The Green's function method reduces solving the PDE to computing integrals over the domain boundary. It naturally incorporates the boundary data and provides an elegant integral representation for the solution.}

The fundamental solution of the Laplace operator in $\mathbb{R}^3$ is
\begin{equation}\label{Phi}
G(x,y)=\frac{1}{4\pi |x-y|} \equiv \Phi(|x-y|), \qquad x\neq y .
\end{equation}
The point $y$ is called the \emph{source point}, and the function $\Phi(|x-y|)$ is referred to as the \emph{source function}.

\subsection{Green's Function for the Half-Space $x_3\ge 0$}\label{HalfSpace}\hfill

Let $x = (x_1, x_2, x_3) \in \Omega$ and $y = (y_1, y_2, y_3) \in \Omega$ be a source point. The Green's function for Laplace's equation in the upper half-space with Dirichlet boundary condition is:
\[
G(x, y) = \frac{1}{4\pi} \left( \frac{1}{|x - y|} - \frac{1}{|x - y^*|} \right),
\]
where $y^* = (y_1, y_2, -y_3)$ is the reflection of $y$ across the plane $x_3 = 0$. This ensures:
\begin{itemize}
    \item $G(x, y) = 0$ for $x \in \partial \Omega$,
    \item $\Delta_x G(x, y) = -\delta(x - y)$ in $\Omega$.
\end{itemize}

Since $\Delta u = 0$, the representation becomes:
\[
u(y) = -\int_{\partial \Omega} \frac{\partial G}{\partial n_x}(x, y)\, g(x)\, dS(x),
\]
where $n_x$ is the outward normal at $x \in \partial \Omega$. Since $\partial \Omega$ is the plane $x_3 = 0$, the outward normal is $n_x = (0, 0, -1)$, so:
\[
\frac{\partial G}{\partial n_x} = -\frac{\partial G}{\partial x_3}.
\]

\subsection*{Normal Derivative of Green's Function}

Differentiating $G(x, y)$ with respect to $x_3$ at $x_3 = 0$, we obtain:
\begin{equation}\label{dGdn}
-\frac{\partial G}{\partial x_3}(x, y)\Big|_{x_3=0} = \frac{y_3}{2\pi \left( |x' - y'|^2 + y_3^2 \right)^{3/2}},
\end{equation}
where $x' = (x_1, x_2)$ and $y' = (y_1, y_2)$. This is known as the \textbf{Poisson kernel} for the 3D half-space:
\[
P(x', y) = \frac{y_3}{\pi \left( |x' - y'|^2 + y_3^2 \right)^{3/2}}.
\]

Substituting into the representation formula:
\[
u(y) = \int_{\mathbb{R}^2} P(x', y)\, g(x')\, dx'.
\]

As $y = (y', y_3) \to (x_0', 0) \in \partial \Omega$, the Poisson kernel $P(x', y)$ becomes sharply peaked around $x' = x_0'$, and tends to a delta function:
\[
\lim_{y_3 \to 0^+} u(y', y_3) = g(y').
\]

Therefore, the function $u(y)$ defined by Green's representation (or equivalently, the Poisson integral) satisfies the Dirichlet boundary condition:
\[
\lim_{y \to x_0 \in \partial \Omega} u(y) = g(x_0).
\]
This confirms that the Green's formula solution matches the prescribed boundary values on $\partial \Omega$.

\iffalse
\subsection{Unit Ball in $\mathbb{R}^3$}\hfill

For the unit ball $\Omega = \{ x \in \mathbb{R}^3 : \|x\| \le 1 \}$, the Green's function is:

\[
G(x, y) = \frac{1}{4\pi |x - y|} - \frac{1}{4\pi |y|\, |x - y^*|}
\]

where $y^* = \frac{y}{|y|^2}$ is the Kelvin inversion of $y$. This ensures $G = 0$ on $\partial \Omega$.
\fi

\subsection{The unit cube}\hfill

For $\Omega=[0,1]^3$,
the Dirichlet Green’s function can be expressed using the method of images as a tri-infinite lattice sum of free-space sources placed at the reflected image points.

\begin{equation}\label{Gcube}
G^\square(x,y)
= \sum_{n\in\mathbb{Z}^3}
\sum_{m\in\{0,1\}^3}
(-1)^{m_1+m_2+m_3}\;
\Phi\!\left(x - y_{n,m}\right),
\end{equation}

where
\[
\Phi(z)=\frac{1}{4\pi\,|z|},
\qquad
y_{n,m}
=\bigl(
(-1)^{m_1} y_1 + 2n_1,\;
(-1)^{m_2} y_2 + 2n_2,\;
(-1)^{m_3} y_3 + 2n_3
\bigr),
\]

and

\[
n=(n_1,n_2,n_3)\in\mathbb{Z}^3,
\qquad
m=(m_1,m_2,m_3)\in\{0,1\}^3.
\]

This yields the Dirichlet Green's function in the unit cube \([0,1]^3\). The series is conditionally convergent.

We denote by $G_{\text{27}}$ the 27 source functions that are closest to the cube:
\begin{equation}\label{G27}
G^\square_{\text{27}}(x,y)
= \sum_{i,j,k=-1}^{1} 
  \sum_{m_1,m_2,m_3=0}^{1} 
  (-1)^{m_1+m_2+m_3}\;
  \Phi\Bigl(
    x - y^{(i,j,k)}_{m_1,m_2,m_3}
  \Bigr),
\end{equation}

\[
y^{(i,j,k)}_{m_1,m_2,m_3}
= \bigl(
(-1)^{m_1} y_1 + 2 i,\;
(-1)^{m_2} y_2 + 2 j,\;
(-1)^{m_3} y_3 + 2 k
\bigr).
\]

\begin{figure}\label{Fig1}
\centering
\includegraphics[width=1.2\linewidth]{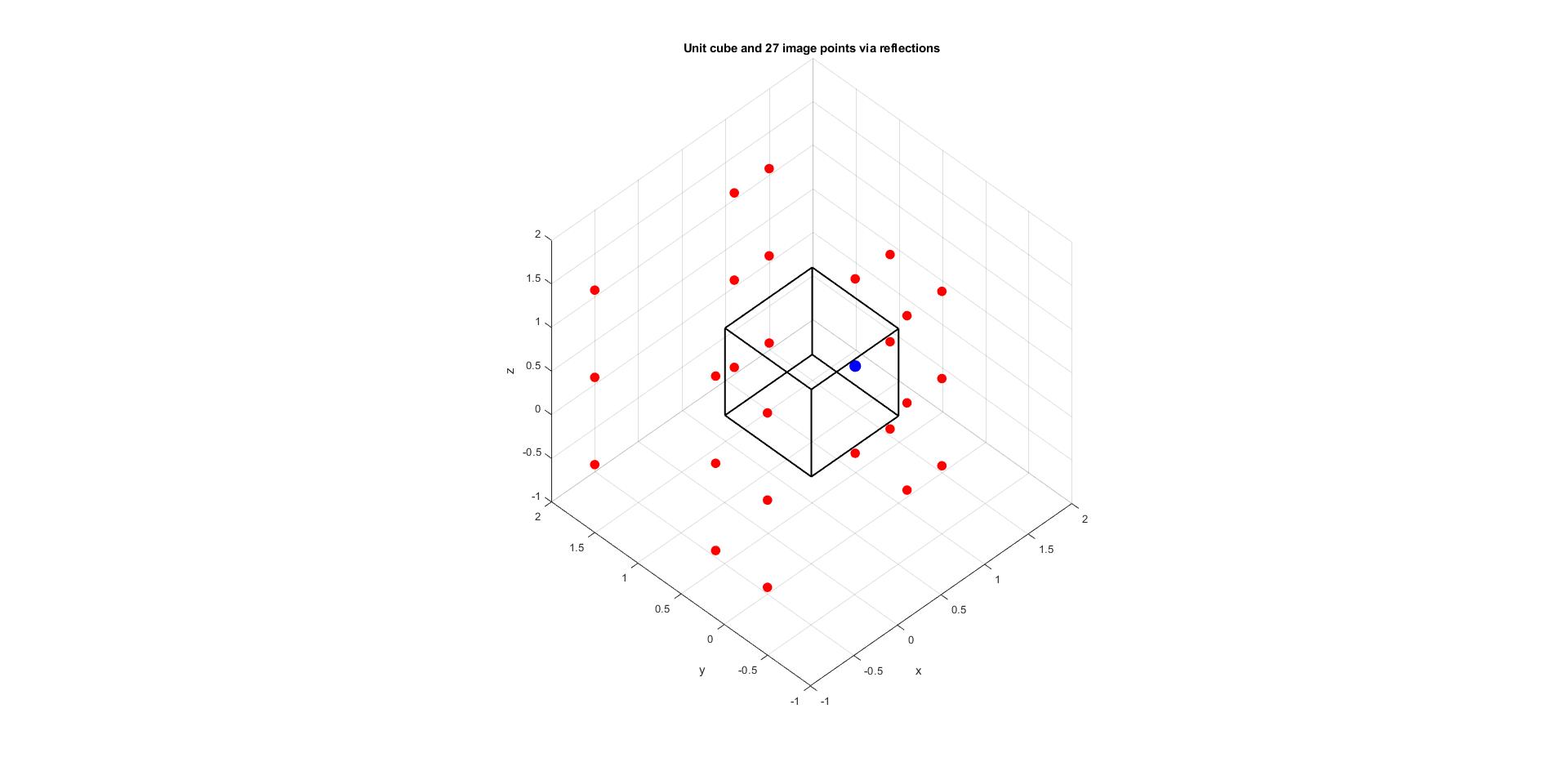}
\caption{The unit cube; The source point in blue plus 26 triple reflections in red.}
\label{G27fig}
\end{figure}

\subsection{The cylinder}\hfill

Let $x=(r,\theta,z)$,\; $y=(r',\theta',z')$, with $0\le r,r'<a$.
Let $\alpha_{n,m}$ be the $m$th positive zero of $J_n$.
The Dirichlet Green's function in the infinite cylinder is
\begin{equation}\label{Gcyl}
G(x,y)
= \frac{1}{2\pi}
  \sum_{n=-\infty}^{\infty}
  e^{\,i n(\theta-\theta')}
  \sum_{m=1}^{\infty}
  \frac{
    J_n\!\left(\tfrac{\alpha_{n,m} r}{a}\right)\,
    J_n\!\left(\tfrac{\alpha_{n,m} r'}{a}\right)
  }{
    \alpha_{n,m}\, a\,
    \bigl[J_{n+1}(\alpha_{n,m})\bigr]^2
  }\;
  e^{-(\alpha_{n,m}/a)\,|z-z'|}.
\end{equation}
It satisfies $G=0$ for $r=a$ and has the free-space singularity
$G(x,y)\sim \frac{1}{4\pi|x-y|}$ as $x\to y$.

\subsection{The finite cylinder}\label{Cyl}\hfill

Let $G_0(x,y)$ denote the Dirichlet Green's function of the infinite
cylinder of radius $1$.
For $y=(r',\theta',y_z)\in\{r'<1,\;0\le y_z\le1\}$, define the axial images
\[
y^{(k)}=(r',\theta',\,y_z^{(k)}),\qquad
y_z^{(k)} = (-1)^k y_z + 2k,\quad k\in\mathbb{Z}.
\]

The \textbf{Dirichlet Green's function} for the finite unit cylinder 
$\Omega = \{ (r, \theta, z) : 0 \le r < 1, 0 \le z \le 1 \}$, 
vanishing on the boundaries $r=1$, $z=0$, and $z=1$, 
is given by the image series
\[
G^\bigcirc(x, y) = \sum_{k \in \mathbb{Z}} (-1)^k 
G_0\bigl(r, \theta, z; r', \theta', y_z^{(k)}\bigr),
\]
where $y_z^{(k)}$ denotes the reflected $z$-coordinates. 
For later use, we introduce the \textbf{three-term image expansion}:
\begin{equation}\label{G3cyl}
G^\bigcirc_3(x; r', \theta', z') = G_0(x; r', \theta', z') 
- G_0(x; r', \theta', -z') 
- G_0(x; r', \theta', 2-z').
\end{equation}

\subsection{Reviewing the Corrected Collocation Method in 2D}\hfill

The paper \cite{Levin1980}  introduces a method 
to improve collocation approximations 
for the harmonic Dirichlet problem by addressing 
the singular behavior of the Green's function.

Standard collocation often fails to produce 
efficient global approximations because 
the distribution of "best" collocation points 
is strongly correlated with the singularities 
of the Green's function, which vary depending 
on the specific point being evaluated.

The paper \cite{Levin1980} suggests overcoming this by:
\begin{itemize}
    \item \textbf{Singular-Regular Decomposition:} 
    Splitting the normal derivative of the 
    Green’s function ($G$) into a \textbf{singular part} 
    ($S$) and a \textbf{regular part} ($R$).
    \item \textbf{Correction Term:} 
    Adding an integral correction term to the 
    initial collocation approximation. 
    This term utilizes the singular part $S$ 
    to "remove" the singular effect.
    \item \textbf{Quadrature Selection:} 
    Choosing collocation points as the abscissae 
    of a best quadrature formula
    for the regular part $R$. This ensures that the 
    approximation is efficient across the 
    entire domain, not just at a single point.
\end{itemize}

The paper proposes a further improvement 
by approximating the remaining error in 
the corrected solution. Since the correction 
term cancels the oscillatory behavior of 
the initial error, this "remainder" can be 
effectively modeled using another round of 
collocation with harmonic polynomials.

Let $\Omega \subset \mathbb{R}^2$ be a bounded domain with a smooth boundary 
$\Gamma$, and let $u$ be the harmonic function satisfying
\[
\Delta u = 0 \quad \text{in } \Omega, 
\qquad 
u = f \quad \text{on } \Gamma .
\]

The procedure in \cite{Levin1980} begins with a collocation approximation of the form
\[
u_h(x,y) = \sum_{j=1}^N a_j \, \phi_j(x,y),
\]
where $\{\phi_j\}$ are smooth trial functions (e.g., harmonic polynomials).
The coefficients $a_j$ are determined by imposing boundary conditions at
collocation points $\{(x_i,y_i)\}_{i=1}^N \subset \Gamma$:
\[
u_h(x_i,y_i) = f(x_i.y_i).
\]
  
The main idea in \cite{Levin1980} is to add a {\em correction term} using the singular part $S$ of the Green’s function
$G$.
The corrected approximation is defined as 
\[
u_h^c(x_0,y_0) 
 = u_h(x_0,y_0) 
   + \int_{\partial\Omega} \frac{\partial S}{\partial n}(x_0,y_0;x,y)\, \big(f(x,y)-\sum_{j=1}^N\alpha_j\phi_j(x,y)\big)\, ds .
\]

In \cite{Levin1980} it was demonstrated that,
for a unit square with jump discontinuities at the corners, the corrected collocation method reduces the error
from $0.5$ to $0.00125$,
using only $16$ collocation points.

The present paper is not just an extension of \cite{Levin1980} to the 3D case. It also makes use of the decomposition of the Green's function into the singular and regular parts, but in a substantially different way.  

\subsection{Challenges in Solving the Unit Cube Problem}\hfill

In many practical applications, 
solving the Laplace equation 
within a cubic domain subject to 
Dirichlet boundary conditions is 
a fundamental task. 

Often, the prescribed boundary data 
are discontinuous, which induces 
singular behavior in the solution 
near the boundaries, specifically 
in the vicinity of edges and corners.

From a numerical perspective, 
classical discretization techniques 
such as the Finite Element Method (FEM) 
can approximate these solutions. 
However, the presence of these 
singularities often leads to a 
deterioration in convergence rates. 
Consequently, achieving high-precision 
results typically requires
intensive local mesh refinement 
or the implementation of specialized 
singular-enrichment techniques. Capturing the correct 
singularities at the corners and edges of the cube is challenging and often leads to 
reduced accuracy if standard basis functions are used. Special techniques, such as 
mesh refinement near singularities or enriched basis functions, are needed to maintain accuracy without excessive computational effort.

Solving the Dirichlet problem for the Laplace equation in a cube,
or more generally in a polyhedral domain,
entails distinctive difficulties.
Even for smooth boundary data,
the solution may exhibit singular behavior near edges and corners.

The paper \cite{Mukh} investigates the high--resolution analysis
of such singularities in three dimensions using the neBEM solver.
The works \cite{Bacu,Cost} study the Sobolev regularity
of harmonic functions in polyhedral domains,
including the cube,
and show that the gradient of the solution typically becomes unbounded
in the vicinity of edges and vertices.
The paper \cite{Ong} develops computational techniques,
based on singular boundary elements,
that are tailored to capture and resolve these corner and edge effects.

We conclude that, in a three-dimensional cube,
the geometry itself may generate singularities,
even when the prescribed boundary data are smooth.
These singularities do not occur in the solution $u$,
which remains continuous,
but rather in its derivatives,
especially near edges and corners.

The present paper is not merely a three-dimensional extension 
of \cite{Levin1980}. 
Although it also relies on a decomposition of the Green’s function 
into singular and regular components, the decomposition is 
exploited here in a fundamentally different manner. 
In the following section, we introduce a two-phase method 
for the high-order evaluation of the solution. 
The resulting approach remains stable and accurate 
even in the presence of sharp singularities.

\section{The S--R method}

\subsection{Singular–Regular Decomposition of the Green’s Function}\hfill

In certain cases, the Green's function is known explicitly, 
allowing the solution to the harmonic Dirichlet problem 
to be computed via boundary integrals using the 
representation formula \eqref{eq:GR}. 
However, for many domains, most notably the unit cube, 
the Green's function is expressed only as a formal series, 
making its practical evaluation non-trivial. 
In the following sections, we adopt the unit cube 
as our primary model setting.

We assume that the Green's function associated with a domain $\Omega$
admits the decomposition
\begin{equation}\label{Dec}
G(x,y)=S(x,y)+R(x,y),
\end{equation}
where $S$ represents the singular part of $G$
and $R$ its regular part.
In what follows, we further assume that $S$ is available
and can be computed with relative ease,
whereas $R$ is not.

\begin{definition}
Let $U \subset \mathbb{R}^n$ be an open domain 
and $\Omega \subset U$ be a closed domain. 
We say that $R(x,y)$ is \textbf{regular of order $m$} 
on $U \times \Omega$ if, for every multi-index $\alpha$ 
with $|\alpha| \le m$, the partial derivative 
$D_x^\alpha R(x,y)$ exists and is bounded 
on $U \times \Omega$.
\end{definition}

\begin{remark}
Note that the definition implies that, 
for each fixed $y \in \Omega$, 
the derivatives $D_x^\alpha R(x,y)$ exist 
and are bounded for all $x \in U$.
\end{remark}

A two-dimensional instance of the decomposition \eqref{Dec}, together with 
its practical use, is given in \cite{Levin1980}. Here, we take as our main example the Green's function for the unit cube, 
introduced above in \eqref{Gcube}.
We decompose $G^\square$ into a singular-regular form as
\begin{equation}
G^\square(x,y) = G^\square_{27}(x,y) + R^\square_{27}(x,y),
\label{SRcube}
\end{equation}
where $G^\square_{27}$ is the partial sum defined in \eqref{G27}, 
and $R^\square_{27}$ is the remainder of the tri-infinite 
expansion \eqref{Gcube}. For any $y \in [0,1]^3$, 
the remainder $R^\square_{27}(x,y)$ is harmonic in $x$ 
throughout the enlarged domain $(-1,2)^3$ and thus 
constitutes the \textbf{regular part} of $G^{\square}$. 
The partial sum $G^\square_{27}(x,y)$ serves as the 
\textbf{singular part}; it contains the primary singularity 
within the unit cube and 26 additional singularities 
in $(-1,2)^3$ arising from triple reflections.

Similarly, the Green's function for the unit cylinder 
$\Omega = \{ (r, \theta, z) : 0 \le r < 1, 0 \le z \le 1 \}$ 
admits the decomposition
\begin{equation}
G^\bigcirc(x,y) = G^\bigcirc_3(x,y) + R^\bigcirc_3(x,y),
\label{SRcyl}
\end{equation}
where the remainder $R^\bigcirc_3$ is harmonic 
in the extended domain $-1 < z < 2$.

\subsection{The Singular-Regular (S--R) approach}\hfill

Assuming the Green's function admits the decomposition \eqref{Dec}, the solution of the Dirichlet problem in $\Omega$ is written as 
\begin{equation}
\label{eq:SR}
u(x) 
= H_S(x)+H_R(x),
\qquad x \in \Omega.
\end{equation}
where
\begin{equation}\label{H_R}
H_R(x)\equiv \int_{\partial\Omega}
    \frac{\partial R}{\partial n_y}(x,y)f(y)\, ds_y .
\end{equation}
and
\begin{equation}\label{H_S}
H_S(x)\equiv \int_{\partial\Omega}
    \frac{\partial S}{\partial n_y}(x,y)f(y)\, ds_y .
\end{equation}

\begin{assumption}\label{Rassumption}
Motivated by the cases of the unit cube and the unit cylinder, 
we assume that for any fixed $y \in \Omega$, 
the regular part $R(\cdot, y)$ is harmonic in an open domain $U$ 
such that $\Omega \subset U$.
\end{assumption}

The following elementary observations constitute the foundation
of the proposed S--R approach.

\begin{corollary}\label{maincor}
In view of Assumption~\ref{Rassumption},
the solution admits the decomposition \eqref{eq:SR} 
where 
\begin{enumerate}
\renewcommand{\labelenumi}{(\Alph{enumi})}
\item $H_R$ is harmonic in $U$.
\item All singular behavior of the solution $u$ is contained in $H_S$.
\end{enumerate}
\end{corollary}

\subsection{The approximation procedure}\hfill

Based on Corollary~\ref{maincor}, we propose the following
approximation procedure.

\begin{enumerate}
\item Select $N$ points $\{x_i\}_{i=1}^N$ distributed over $\partial\Omega$.

\item Evaluate
\begin{equation}\label{tildeHS}
H_S(x_i)=\int_{\partial\Omega}
    \frac{\partial S}{\partial n_y}(x_i,y)f(y)\, ds_y,
\qquad i=1,\ldots,N,
\end{equation}
using a suitable quadrature rule, denote the resulting approximations by $\tilde H_S(x_i)$.

\item Compute
\begin{equation}\label{HRtilde}
\tilde H_R(x_i)=u(x_i)-\tilde H_S(x_i),
\qquad i=1,\ldots,N.
\end{equation}

\item Construct a harmonic approximation $P_N$ of $H_R$
from the data $\{\tilde H_R(x_i)\}_{i=1}^N$.

\item For any given $x\in\Omega$, approximate $H_S(x)$
by $\tilde H_S(x)$ using a suitable quadrature rule.

\item Define the approximation $u_N$ to $u$ at $x\in\Omega$ by
\begin{equation}\label{uN}
u_N(x)=\tilde H_S(x)+P_N(x).
\end{equation}
\end{enumerate}

\begin{observation}
Once the global harmonic approximation of $H_R$ is available, 
the evaluation of $u$ at any prescribed point $x$ 
reduces to a separate and independent computation.
\end{observation}

\begin{observation}
The approximation $u_N$ in \eqref{uN} consists of a harmonic
approximation $P_N$ of $H_R$ and a component $\tilde H_S$,
which is generally non-harmonic. Indeed, $\tilde H_S(x)$ is
typically not harmonic in $x$, since the quadrature rule
used in step~(5) above depends on the evaluation point $x$.
\end{observation}

In the sequel, we examine the individual steps of the approximation
procedure, with particular emphasis on the unit cube.

\subsection{Theoretical and implementation details for the unit cube}\hfill

For the unit cube, we recall that the singular part is 
$S = G_{\square}^{27}$, as defined in \eqref{G27}, 
and the corresponding regular part is 
$R = R_{\square}^{27}$, see \eqref{SRcube}. 
As shown below, $R(\cdot,y)$ is harmonic in the enlarged domain 
$(-1,2)^3$. 
We now prove that the associated regular component of the solution 
is likewise harmonic in $(-1,2)^3$.

\begin{proposition}
Let $\Omega=[0,1]^3$ and consider the decomposition
\eqref{SRcube}, where $G_{\square}$ is given by the lattice
representation \eqref{Gcube} and $G_{\square}^{27}$ by the
truncated sum \eqref{G27}.  
Then, for every $y\in\Omega$, the remainder
$R_{\square}^{27}(\cdot,y)$ is harmonic in the enlarged domain
$(-1,2)^3$.  Consequently, the function
\[
H_R(x)=\int_{\partial\Omega}
\frac{\partial R_{\square}^{27}}{\partial n_y}(x,y)\,
f(y)\,dS_y,
\]
defined in \eqref{H_R}, is harmonic in $(-1,2)^3$.
\end{proposition}

\begin{proof}
By \eqref{Gcube}, the Green’s function is represented as a
lattice sum of free-space source functions
$\Phi(x-y_{n,m})$, while \eqref{G27} retains only those
terms with $n\in\{-1,0,1\}^3$.
Hence, in the decomposition \eqref{SRcube}, the remainder
$R_{\square}^{27}$ consists precisely of those image terms
with $n\notin\{-1,0,1\}^3$.
For any $y\in[0,1]^3$, such image points $y_{n,m}$ lie
outside the box $(-1,2)^3$.
Therefore, each term
$x\mapsto\Phi(x-y_{n,m})$ is harmonic throughout
$(-1,2)^3$.
On compact subsets of $(-1,2)^3$, the series defining
$R_{\square}^{27}(\cdot,y)$ admits termwise differentiation
with respect to $x$ (for instance, via uniform convergence
of the second-derivative series),
and thus $\Delta_x R_{\square}^{27}(x,y)=0$
in $(-1,2)^3$.
Differentiation with respect to $y$ does not affect
harmonicity in $x$, and hence
$\partial R_{\square}^{27}/\partial n_y(\cdot,y)$
is harmonic in $x$ on $(-1,2)^3$.
Finally, the representation \eqref{H_R}
shows that $H_R$ is an integral of harmonic functions
with respect to $y$, and is therefore harmonic
in $(-1,2)^3$.
\end{proof}

\subsubsection{Computing $\tilde H_R$ on $\partial\Omega$}\label{HRxi}\hfill

We begin with Step~(2), namely the computation of
$\tilde H_S(x_i)$ for $x_i \in \partial\Omega$.

As an illustrative example, consider a point $x_i$
located on the upper face of the cube, namely $z=1$.
All source points in \eqref{G27} that lie in the plane $z=1$
contribute exactly the value $f(x_i)$ to $H_S(x_i)$.
This follows from the analysis presented in
Section~\ref{HalfSpace}.
The remaining nine reflection points lie in the plane $z=-1$.
Consequently, the corresponding normal derivatives are regular 
along $\partial\Omega$, and their contribution to the integral 
in \eqref{tildeHS} can therefore be evaluated to any prescribed 
accuracy using standard quadrature rules. 
By \eqref{HRtilde}, the resulting quantity provides 
the desired approximation of $H_R$.

\subsubsection{Choosing the collocation points}\hfill

The choice of the points $\{x_i\}_{i=1}^N$ in Step~(1)
has a direct impact on the quality of the harmonic interpolant $P_N$
constructed in Step~(4).
We examined two alternatives.
The first employs points taken from a uniform grid on each face of the cube,
whereas the second uses a tensor--product grid of Gaussian points.
Although both choices produced approximations of comparable accuracy,
the uniform distribution yielded a markedly better conditioned
interpolation matrix.
This improvement in conditioning translated into enhanced numerical stability
of the overall procedure.

Assuming $N=6n^2$, we distribute $n^2$ points on each face of the cube.
For instance, for the face on $z=0$, 
the uniform placement is given by $\left({(2i-1)}/{2n},\,{(2j-1)}/{2n},0\right), i,j=1,\ldots,n.$ 
The remaining faces are treated analogously.

\subsubsection{The harmonic approximation $P_N$.}\label{PN}\hfill

Assume that the values $\{\tilde H_R(x_i)\}_{i=1}^N$
are computed with an approximation error of order $\varepsilon$.
Our goal is to construct a harmonic approximation $P_N$
to $H_R$ based on this data. We investigated two alternative approximations:
$P_N^{\mathrm{I}}$, obtained by harmonic polynomial interpolation,
and $P_N^{\mathrm{II}}$, constructed via the Method of Fundamental Solutions.
We evaluated their performance in terms of accuracy and numerical stability.

\noindent\textbf{The approximation $P_N^{\mathrm{I}}$:}

A natural first candidate is to construct the approximation by interpolation with harmonic polynomials.
In this approach, we rely on the following classical approximation result
\cite{Walsh,dai2013}.

Let $\Omega$ be a bounded domain in $\mathbb{R}^3$
(e.g., the unit cube centered at the origin).
Assume that $u$ is harmonic in an open neighborhood
containing a ball $B_R(0)$ that strictly contains $\Omega$.
For the unit cube centered at the origin,
\[
\Omega \subset B_\rho(0),
\qquad \rho=\frac{\sqrt{3}}{2}.
\]
 In our case, since $H_R$ is harmonic in $(-1,2)^3$,
after translating the unit cube to be centered at the origin,
the function remains harmonic in a neighborhood
containing a ball of radius $R= 3/2$.

Expand $u$ in solid spherical harmonics about the origin:
\[
u(x)
=
\sum_{\ell=0}^{\infty}
\sum_{m=-\ell}^{\ell}
a_{\ell m}\,
r^{\ell}\,
Y_\ell^{m}(\theta,\phi),
\qquad r=|x|.
\]

Let $Q_N$ denote the truncation of this expansion to
$\ell\le N$, which is a harmonic polynomial of degree $N$.
Then, for any $0<\rho<R$,
\[
\|u-Q_N\|_{L^\infty(B_\rho(0))}
\le
C(u,R)
\left(\frac{\rho}{R}\right)^{N+1}.
\]

Thus, the approximation error decays geometrically in $N$.
The constant $C(u,R)$ depends on an upper bound of $u$
on $B_R(0)$. By the maximum principle, 
as the difference $u - Q_N$ is itself harmonic, 
the maximum error on the interior is bounded by its values on the boundary.

In the present setting, we do not employ a spherical harmonic expansion.
Instead, we construct a harmonic polynomial approximant of $H_R$
by interpolating the approximate data
$\{\tilde H_R(x_i)\}_{i=1}^N$, and we denote it as $P_N^\mathrm{I}$.
If $H_R$ extends harmonically to a neighborhood strictly larger than
$\Omega$, its polynomial approximations exhibit a geometric rate of convergence. 
Provided that the interpolation procedure is stable and that the data
$\tilde H_R(x_i)$ approximate $H_R(x_i)$ with sufficiently high order,
the resulting interpolant $P_N^\mathrm{I}$ retains the same geometric rate of
convergence, up to the additional error introduced by the data
approximation.

\noindent\textbf{The approximation $P_N^\mathrm{II}$}: 

We further investigated an alternative interpolation strategy
for approximating $H_R$.
In this approach, the harmonic approximant is represented
as a superposition of source functions.
For clarity, we describe the method in the case of the unit cube
centered at the origin, $[-0.5,0.5]^3$.
Assume that the values of a harmonic function $u$
are given at points $\{x_i\}_{i=1}^N.$ distributed on the faces of the cube. Define a set of source points, 
\begin{equation}\label{alpha}
p_i=\alpha x_i,\ \ i=1,...,N.\ \ \alpha>1.
\end{equation}

Figure~\ref{PSource} illustrates a representative configuration
of 150 collocation points on the faces of the cube
and their associated 150 source points, with $\alpha=3$.

\begin{figure}
\centering
\includegraphics[width=1.1\linewidth]{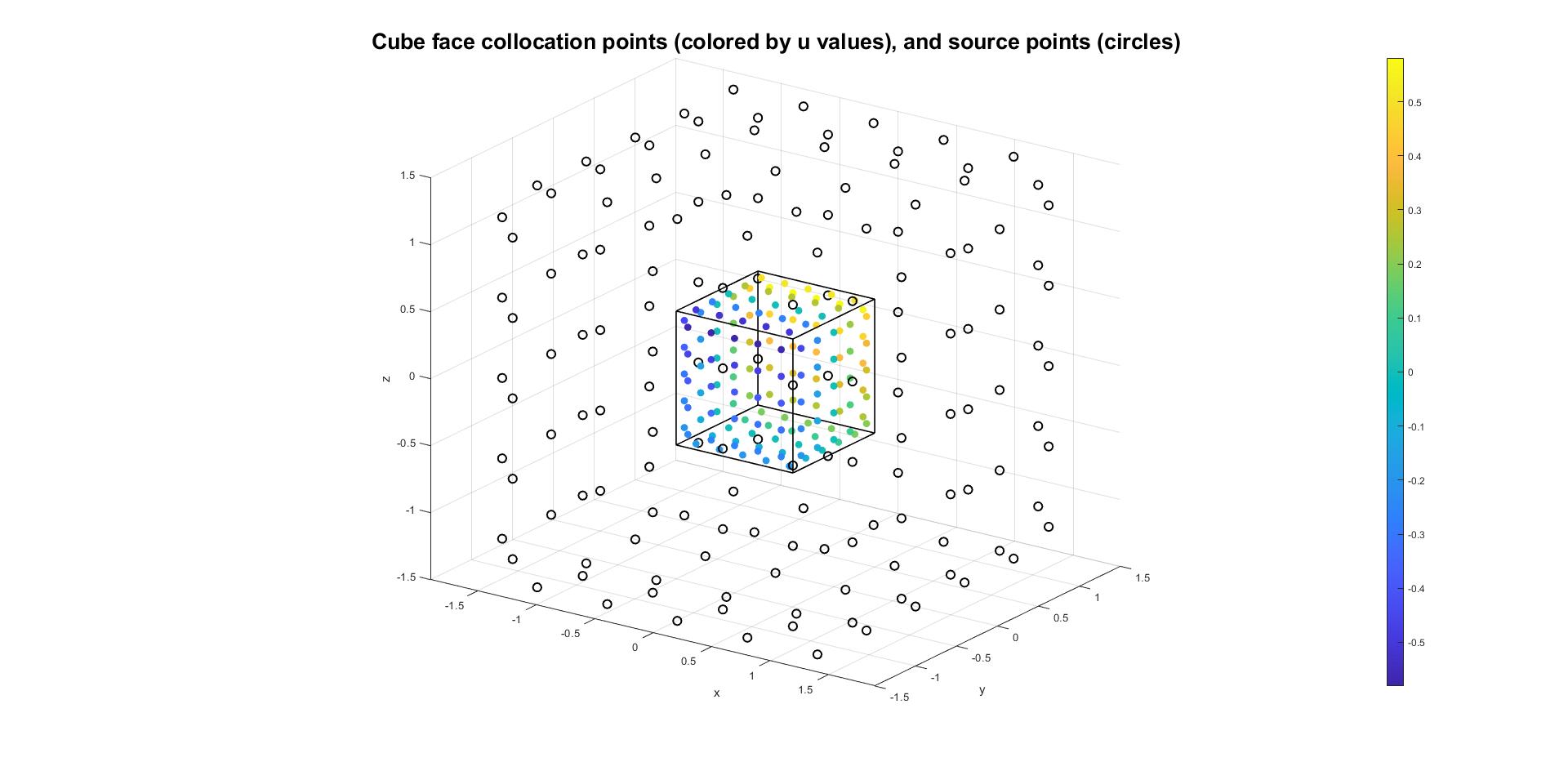}
\caption{The figure shows the 150 interpolation points on the cube's faces, colored by their sampled values, along with the corresponding source points (marked by circles) used in the interpolation basis.}\label{PSource}
\end{figure}
We seek an approximation $P_N^\mathrm{II}$ of the form
\begin{equation}\label{Phisum}
P_N^\mathrm{II}(x)=\sum_{j=1}^N c_j\,\Phi\bigl(|x-p_j|\bigr),
\end{equation}
where $\Phi$ is the source function defined in \eqref{Phi}.
The coefficients $\{c_j\}_{j=1}^N$ are determined by enforcing the
interpolation conditions
\begin{equation}\label{Phiinterpolation}
P_N^\mathrm{II}(x_i)=\sum_{j=1}^N c_j\,\Phi\bigl(|x_i-p_j|\bigr),
\qquad i=1,\ldots,N.
\end{equation}

The representation \eqref{Phisum} corresponds to the classical
Method of Fundamental Solutions (MFS); see, e.g.,
\cite{Fairweather,Barnett}.
At the continuous level, this ansatz can be viewed as a discrete
single--layer potential with sources located outside the domain,
and its validity is supported by the potential-theoretic
foundations of harmonic boundary value problems.
At the discrete level, the resulting interpolation matrix
$A_{ij}=\Phi(|x_i-p_j|)$ is typically invertible for generic
placements of collocation and source points; however,
invertibility is not guaranteed in general, and the system
may become severely ill--conditioned, particularly when the
sources approach the boundary.

\noindent\textbf{Comparing $P_N^\mathrm{I}$ and $P_N^\mathrm{II}$}:

We conducted a series of numerical experiments
to identify the method most suitable for our specific application.
Given a harmonic function defined on $(-1.5,1.5)^3$,
the experiments examined its approximation on the unit cube
$[-0.5,0.5]^3$. We present the results for the test case with 150 data points,
distributed on a uniform grid over each of the six faces of the unit cube. The source points for $P_N^\mathrm{II}$ are defined by \eqref{alpha} with $\alpha=3$.
Both approximation methods were applied to the following two harmonic
test functions:
\[
u_1(x,y,z)=\cos(0.6x)\sin(0.8y)\exp(z),
\]
\[
u_2(x,y,z)=\frac{1}{\sqrt{x^2+y^2+(z-1.6)^2}}.
\]
We note that $u_1$ is harmonic throughout $\mathbb{R}^3$,
whereas $u_2$ possesses a singularity located near the cube
$[-1.5,1.5]^3$.

\medskip
\begin{table}[ht]
\centering

\renewcommand{\arraystretch}{1.4}      % row spacing
\setlength{\tabcolsep}{14pt}           % column spacing

\begin{tabular}{lccc}
\hline
Method & $\mathrm{cond}(A)$ & Error in $u_1$ & Error in $u_2$ \\
\hline
$P^{\mathrm{I}}$  & $7.61\times 10^{17}$ & $8.13\times 10^{-6}$ & $8.26\times 10^{-4}$ \\
$P^{\mathrm{II}}$ & $4.25 \times 10^{8}$ & $1.83\times 10^{-5}$ & $1.07\times 10^{-5}$ \\
\hline
\end{tabular}

\caption{Condition number and approximation errors for the two test functions.}
\label{tab:comparison}
\end{table}

Table~\ref{tab:comparison} summarizes the conditioning of the
interpolation matrix and the corresponding approximation errors
for the two methods.
The polynomial interpolation approach $P^{\mathrm{I}}$
leads to an extremely ill--conditioned system
($\mathrm{cond}(A)\approx 7.61\times 10^{17}$),
yet it achieves high accuracy for the globally smooth test function $u_1$.
However, its performance deteriorates significantly for $u_2$,
whose nearby singularity reduces the effective radius of harmonic continuation,
resulting in an error of order $10^{-4}$.
In contrast, the method of fundamental solutions $P^{\mathrm{II}}$
produces a substantially better conditioned system
($\mathrm{cond}(A)\approx 4.25\times 10^{8}$)
and delivers consistently accurate results for both test functions,
with errors of order $10^{-5}$.
These results indicate that $P^{\mathrm{II}}$ is more robust, 
especially in the presence of nearby singularities. 
When the number of data points was increased to $49$ on each face of the cube, 
the approximation error of $P^{\mathrm{II}}$ for $u_2$ decreased to 
$4.15 \times 10^{-7}$. 
However, this improvement came at the cost of a larger 
condition number, with $\mathrm{cond}(A) = 3.34 \times 10^{12}$, 
indicating increased ill-conditioning of the system.

\subsubsection{Numerical error assessment for harmonic approximations of $H_R$}\label{ErrorTest}\hfill

The numerical experiments above demonstrate the performance of
the harmonic approximations $P^{\mathrm{I}}$ and $P^{\mathrm{II}}$
on prescribed test functions. We now turn to the more delicate task of deriving quantitative
error bounds for these methods when applied to the unknown
harmonic component $H_R$. We emphasize that such an estimate
is essential for establishing a reliable global error bound
for the S--R method.

Our analysis is based on the observation in Section~\ref{HRxi} that
the boundary values of $H_R$ on the faces of the cube can be
evaluated to arbitrary accuracy by means of standard quadrature
schemes. For the approximation error estimation, we apply the following procedure:
\begin{itemize}
\item Augment the interpolation
set $\{x_i\}_{i=1}^{N}$ with a collection of reference points
$\{q_j\}_{j=1}^{J}$, uniformly distributed over all six faces of the
cube. Figure~\ref{Figure4}(A) illustrates a representative
configuration with $N=150$, corresponding to 25 interpolation
points on each face. The interpolation points $\{x_i\}$ are shown in
red, while the reference points $\{q_j\}$ are displayed in green.
\item
Compute a high-accuracy approximation of $\tilde H_R\sim H_R$ at both
sets of points, namely $\{x_i\}$ and $\{q_j\}$.
\item
Using the values $\{\tilde H_R(x_i)\}_{i=1}^N$
, compute the harmonic approximation  $P_N\sim H_R$ in the cube.
\item Estimate the error bound for the approximation $\tilde H_R$ over the cube as 

\begin{equation}\label{Bound}
E^N_R= 2E^N_{max}
\end{equation}
where
\begin{equation}\label{Emax}
E^N_{max}= max_{j=1}^J|\tilde H_R(q_j)-P_N(q_j)|.
\end{equation}

\end{itemize}

The error estimate \eqref{Bound} relies on the following observation:

\begin{observation}
Let the interpolation points $\{x_i\}$ be distributed on a uniform
mesh of size $h$, and let the reference points $\{q_j\}$ lie on a
uniform mesh of size $h/2$, so that each $q_j$ is located at the
midpoint between two adjacent interpolation points, as illustrated
in Figure~\ref{Figure4}(A).
Then,
\begin{equation}\label{Max}
\max_{x\in [0,1]^3} \bigl|\tilde H_R(x)-P_N(x)\bigr|
= E_{max}(1 + O(h^2)).
\end{equation}
The estimate is first established for the maximum error on the faces
of the cube. The global bound \eqref{Max} then follows from the
maximum principle for harmonic functions.

\end{observation}

\subsubsection{A Chebyshev spectral alternative for approximating $H_R$.}\label{Cheby}\hfill

As an alternative to harmonic interpolation by polynomials or source
functions, the regular component $H_R$ may be approximated by a
tensor–product Chebyshev spectral collocation method on the cube
$\Omega=[0,1]^3$.

Let
$
x_i=\tfrac12\!\left(1+\cos(i\pi/N)\right),
\ i=0,\dots,N,
$
be the Chebyshev–Gauss–Lobatto nodes mapped to $[0,1]$, and define the
three–dimensional grid by the tensor product $(x_i,y_j,z_k)$.
The Laplace equation is enforced at interior collocation points using the
standard Chebyshev differentiation matrices in each coordinate
direction, while the Dirichlet boundary conditions are imposed at the
boundary nodes on the six faces of the cube; see, e.g.,
Canuto–Hussaini–Quarteroni–Zang~\cite{Canuto}.

If $H_R$ extends harmonically to a neighborhood strictly larger than
$\Omega$, as is the case for the cube, where $H_R$ is harmonic in
$(-1,2)^3$, the Chebyshev approximation converges spectrally (i.e.,
exponentially) with respect to $N$. In order to apply this approach,
the boundary values of $H_R$ must be supplied at the $(N+1)^2$
Chebyshev–Gauss–Lobatto nodes on each face; as explained in
Section~3.4.1, these values can be computed to any prescribed accuracy
using standard quadrature rules.

We note that a detailed numerical comparison between the spectral
collocation approach and the source–based approximation lies beyond
the scope of the present work.

\subsubsection{Computing $H_S$ inside $\Omega$}\label{CompHS}\hfill

Using \eqref{dGdn} together with the definition of $H_S$ in \eqref{H_S},
the evaluation of $H_S$ reduces to the approximation of integrals of the form
\begin{equation}
\iint_{[0,1]^2}
\frac{g(x,y)}{\big((x-a)^2+(y-b)^2+d^2\big)^{3/2}}
\, dx\, dy,
\qquad (a,b)\in [0,1]^2.
\label{Intabd}
\end{equation}
The main difficulty arises when $d$ is very small. Although the integral
remains finite for $d>0$, the kernel becomes sharply peaked near $(a,b)$,
and its derivatives grow rapidly in a neighbourhood of this point.
Consequently, standard tensor–product quadrature rules converge slowly
and may fail to deliver sufficient accuracy.

An efficient approach for treating nearly singular integrals 
is the self-adaptive coordinate transformation proposed by 
Telles~\cite{Telles}. 
This method applies a nonlinear change of variables 
that redistributes the quadrature nodes to cluster them 
near the near-singular location, thereby significantly enhancing accuracy. 
As a consequence, standard Gauss quadrature can still be employed, 
while the effective resolution is automatically increased 
in the region where the integrand exhibits rapid variation. 
This makes the method both efficient and robust for small values of $d$.

An alternative approach, in the same spirit but simpler to implement, 
is to apply standard quadrature rules on a locally refined mesh. 
For example, consider the integral in~\eqref{Intabd} with $(a,b)=(0.2,0.2)$ 
and small $d$. 
We subdivide $[0,1]^2$ into nine rectangles and apply a standard 
tensor-product quadrature rule on each subregion, 
using a finer mesh in the vicinity of the near-singular point. Figure \ref{Figure4}(B) exhibits the mesh refinement near $(0.2,0.2)$.

\begin{figure}[ht]
\centering

\begin{subfigure}{0.48\textwidth}
    \centering
\includegraphics[width=1.2\linewidth]{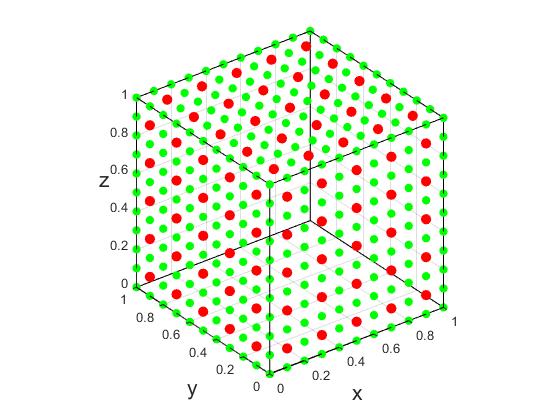}
    \caption{Face points}
    \label{fig:sub1}
\end{subfigure}
%\hfill
\begin{subfigure}{0.48\textwidth}
    \centering
    \includegraphics[width=1.1\linewidth]{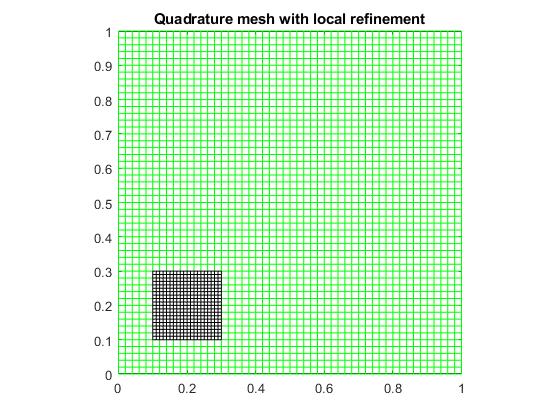}
    \caption{Two meshes}
    \label{fig:sub2}
\end{subfigure}

\caption{}
\label{Figure4}
\end{figure}

For the finite cylinder considered in Section~\ref{Cyl}, the singular
part of the Green's function is given by~\eqref{G3cyl}, where $G_0$ is defined by~\eqref{Gcyl} with $a=1$.
We emphasize that the infinite series in~\eqref{Gcyl} decays rapidly,
owing to the exponential factor in its terms, and therefore converges
very quickly in practical computations. The computation of the corresponding $H_S$ can be performed
using the same procedure described above for the unit cube.

\section{Numerical experiments}\label{sec:NumExp}

We consider the Dirichlet problem for the Laplace equation
in the unit cube $\Omega=(0,1)^3$,
\begin{equation}\label{eq:LaplaceCube}
\Delta u = 0 \quad \text{in }\Omega,
\qquad
u=g \quad \text{on }\partial\Omega,
\end{equation}
with piecewise constant boundary data, $g\equiv 1$ on the upper face and $g\equiv 0$ on all other faces.
This test problem is intentionally non-smooth at the edges
adjacent to the upper face.
Although the solution $u$ is continuous in $\overline{\Omega}$,
its derivatives exhibit singular behavior near the edges and corners,
which makes it a sensitive benchmark for accuracy and stability.

The test problem admits two standard physical interpretations.
First, it models the electrostatic potential inside a perfectly conducting
cubical cavity in which the upper face is maintained at a fixed potential
$V_0$ and the remaining five faces are grounded. In the absence of free
charge, the potential satisfies Laplace’s equation with Dirichlet boundary
data of the form considered here; see, e.g.,
\cite{MorseFeshbach}.
Equivalently, the problem represents the steady-state temperature
distribution in a homogeneous cube with no internal heat sources,
where the upper face is held at temperature $T_0$ and the other faces
are maintained at zero temperature. The steady temperature then solves
Laplace’s equation with the same Dirichlet boundary conditions;
see, for example, \cite{Evans}.

The S--R method was applied to the above test problem.
We computed the values $\tilde H_S(x_i)$ as described in Section \ref{HRxi}.
The regular component of the solution was approximated by a harmonic
representation in the form of a linear combination of source functions,
that is, $P_N = P_N^{\mathrm{II}}$ with $N=150$. The surface values of $P_{150}$ are displayed in Figure \ref{P_150}.

\begin{figure}
\centering
\includegraphics[width=1.1\linewidth]{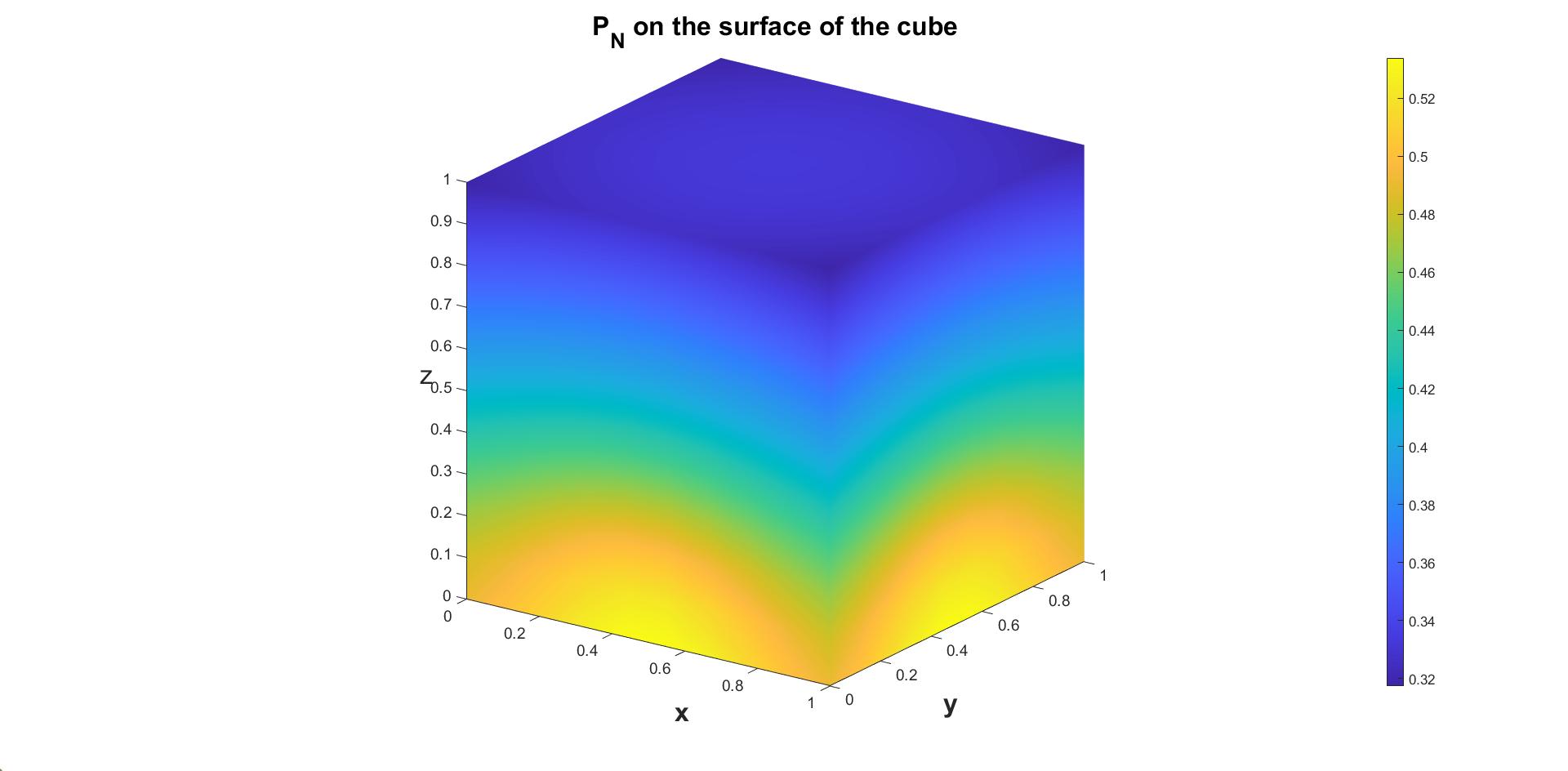}
\caption{The computed approximation $P_{150}$ of the regular component of the solution.}\label{P_150}
\end{figure}

With $P_N$ available, and after computing $\tilde H_S$, the approximation $u_N$ takes the form
\begin{equation}\label{uN2}
u_N(x) = \tilde H_S(x) + P_N(x), \qquad x \in \Omega.
\end{equation}

To visualize the singular behavior of the solution,
we plot in Figure~\ref{corner} the solution values along a planar
cross-section of the cube near the corner $(0,0,1)$,
where the singularity is most pronounced.
The triangular planar section lies at a distance of $0.0866$ from the corner,
and the solution values $u$ are visualized using a color-coded representation. The Figure clearly captures the fast singular transition from the value $1$ on the upper face
to the value $0$ on the adjacent faces.

\begin{figure}
\centering
\includegraphics[width=1.1\linewidth]{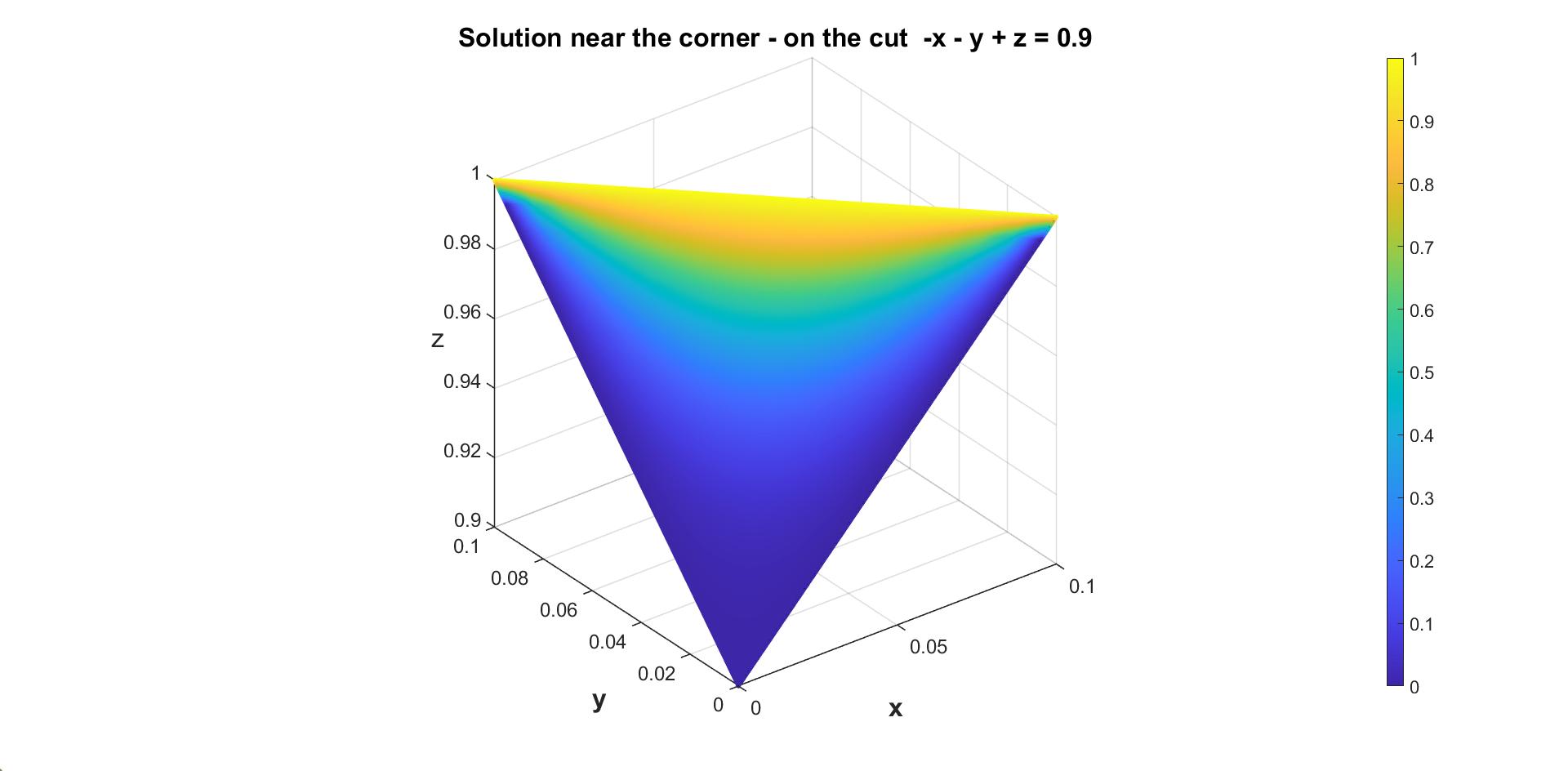}
\caption{The figure displays the computed approximation in the vicinity of the cube corner.
It clearly captures the singular transition from the value $1$ on the upper face
to the value $0$ on the adjacent faces.}\label{corner}
\end{figure}

\subsection{The approximation error}\hfill

The accuracy of $u_N$ is affected by three separate sources of error. The first source of error stems from the computation of 
$\tilde H_S(x_i)$ for $x_i \in \partial\Omega$. 
As discussed in Section~\ref{HRxi}, this step reduces to numerically integrating a regular function. 
Therefore, $H_S(x_i)$, and consequently $H_R(x_i)$, can be evaluated to any desired accuracy using standard quadrature schemes. The second source of error is associated with the reconstruction 
of the regular harmonic function from its boundary values. 
This is discussed in Section~\ref{PN}. 
It is shown there that approximation by a linear combination 
of source functions ($P_N^{\mathrm{II}}$, see \eqref{Phiinterpolation})
is stable and provides high accuracy for harmonic functions 
whose region of regularity is comparable to that of $H_R$. As shown in Section~\ref{Bound}, a reliable estimate $E_R$
for the approximation error of $H_R$ throughout the cube can be derived. For the test problem presented above, we obtained the error bound $E^{150}_R=0.22\times 10^{-5}$ for the approximation of $H_R$. Increasing the number of interpolation points to 

The third source of error arises from the approximation 
of the surface integrals involved in the computation of $H_S(x)$. 
If $x$ is sufficiently far from $\partial\Omega$, 
standard tensor-product quadrature rules yield high-order accuracy. 
When $x$ is close to $\partial\Omega$, 
we employ tensor-product Simpson quadrature on a rectangular 
subdivision of the faces, with additional refinement 
near the near-singular point, as described in Section~\ref{CompHS}. 

This strategy was employed to generate the results near the corner 
of the cube shown in Figure~\ref{corner}, 
using a mesh refinement adapted to the distance of $x$ 
from the boundary. 
The computed values were further validated by additional 
grid refinement. 
Overall, a maximum error of $0.5\times 10^{-5}$ 
was achieved for the approximation of $H_S$ over a uniform distribution of points 
in the triangular region depicted in Figure~\ref{corner}.

Taking into account the error bound obtained for the approximation of $H_R$, and combining it with \eqref{eq:SR} and \eqref{uN2}, we obtain that over the corner triangle shown in Figure \ref{corner} the approximation error satisfies
\begin{equation}
    |u(x)-u_{150}(x)| \le 0.5\times 10^{-5} + 2.2\times 10^{-5}
    = 2.7\times 10^{-5}.
\end{equation}

\end{document}